\documentclass[11pt,a4paper]{article}
\usepackage{graphicx}
\usepackage{indentfirst}
\usepackage{mathrsfs}
\usepackage{fancyhdr}
\usepackage{color}
\usepackage{amsfonts,amssymb,amsmath,amsthm}
\usepackage{authblk}
\usepackage[unicode,pdftex]{hyperref}
\hypersetup{
	colorlinks = true,
	citecolor = green,
	anchorcolor = blue,
	linkcolor = blue
}

\newtheorem{thm}{Theorem}[section]
\newtheorem{lemma}[thm]{Lemma}

\newtheorem{prop}[thm]{Proposition}
\newtheoremstyle{rem}{10pt}{10pt}{\rmfamily}{}{\bfseries}{.}{.5em}{} 
\theoremstyle{rem}
\newtheorem{rem}[thm]{Remark}

\numberwithin{equation}{section}

\title{Strichartz estimates to fractional Schr\"{o}dinger equations}
\author[1\footnote{Corresponding author. E-mail address: jiechern@jmu.edu.cn}]{Jie Chen}
\author[1]{Yiran Gong}
\author[2]{Ying Zhang}
\affil[1]{\scriptsize \textit{School of Science, Jimei University, Xiamen 361021, P.R. China}}
\affil[2]{\scriptsize \textit{Academy of Mathematics and Systems Science, CAS, Beijing 100190, P.R. China}}
\date{}

\begin{document}
	\maketitle
	\begin{abstract}
		In this paper, we firstly study Strichartz estimates
		$$\|e^{it D^\alpha} u_0\|_{L_t^q(\mathbb{R};L_x^r(\mathbb{R}^d))}\leq C(d,\alpha,q,r,s)\|u_0\|_{\dot{H}^s}.$$
		We show some counterexamples for  $(q,r) = (2,\infty)$. Then we consider the embedding $X^{s,b}_\alpha \hookrightarrow L_t^q(\mathbb{R};L_x^r(\mathbb{R}^d))$, where $\|u\|_{X_\alpha^{s,b}}:=\|\langle\xi\rangle^s\langle \tau-|\xi|^\alpha\rangle^b\hat{u}(\tau,\xi)\|_{L^2_{\tau,\xi}}$. We present the necessary and sufficient conditions for this embedding.
	\end{abstract}
	
	\section{Introduction}
	Consider the fractional Schr\"{o}dinger equation $\partial_tu = iD^\alpha u$ where $D^\alpha:=\mathscr{F}_x^{-1}|\xi|^\alpha \mathscr{F}_x$. If $\alpha = 1$, the equation is referred as half-wave equation. If $\alpha = 2$, it is the Schr\"{o}dinger equation. For general $0<\alpha<2$, it also appears in quantum  mechanics. See for example Laskin \cite{laskin2000fractional1,laskin2000fractional2}. For $\alpha>2$, especially $\alpha = 4$, there is also a large body of research. See Pausader \cite{pausader2009cubic} and reference therein.
	
	The classical Strichartz estimate is
	\begin{equation}\label{clastri}
		\|e^{itD^\alpha}u_0\|_{L_t^q(\mathbb{R};L_x^r(\mathbb{R}^d))}\leq C(d,\alpha,q,r,s) \|u_0\|_{\dot{H}^s}
	\end{equation}
	for some $q,r,s$. Strichartz found the relation between this estimate and the Fourier restriction problem in \cite{strichartz1977restrictions} when $q = r$. The estimates were later expanded by multiple researchers. Specifically, for the endpoint case, see Keel--Tao \cite{keel1998endpoint}. These estimates are very useful for studying the well-posedness of nonlinear dispersive equation, and one may refer to \cite{tao2006nonlinear}.
	
	The corresponding Bourgain space (cf. \cite{bourgain1993fourier}) for the fractional Schr\"{o}dinger equation is
	\begin{equation*}
		\|u\|_{X_\alpha^{s,b}}:=\|\langle\xi\rangle^s\langle \tau-|\xi|^\alpha\rangle^bu(\tau,\xi)\|_{L^2_{\tau,\xi}}
	\end{equation*}
	where $\langle\cdot\rangle:=(1+|\cdot|^2)^{1/2}$. Our notation for the Bourgain space doesn't specify the dimension $d$, but it can be easily determined from the context. This space is useful in proving local well-posedness of dispersive equations. The relation between the Strichartz estimate \eqref{clastri} and the Bourgain space is: If \eqref{clastri} holds, then
	\begin{equation}\label{stritypeesti}
		\|u\|_{L_t^q(\mathbb{R};L_x^r(\mathbb{R}^d))}\leq C(d,\alpha,q,r,s,b)\|u\|_{X_\alpha^{s,b}}
	\end{equation}
	holds for any $b>1/2$. See for example \cite{erdougan2016dispersive}.
	
	In this paper, we focus on the necessary and sufficient conditions for \eqref{clastri}--\eqref{stritypeesti} to hold. Firstly we have
	\begin{thm}\label{classicalstri}
		Let $d \in \mathbb{N}$, $\alpha>0$, $s\in \mathbb{R}$, $0<q,r\leq \infty$. Then \eqref{clastri} holds for any $\mathscr{F}u_0\in C_0^\infty(\mathbb{R}^d\setminus \{0\})$ iff $2\leq q,r\leq \infty$, $s = d/2-\alpha/q-d/r$, $(q,r)\neq (\infty,\infty)$, $(2,\infty)$, and
		\begin{itemize}
			\item $\alpha\neq 1$, $2/q+d/r\leq d/2$, or
			\item $\alpha = 1$, $(d,q)\neq (2,4)$, $2/q+(d-1)/r\leq (d-1)/2$.
		\end{itemize}
	\end{thm}

	\begin{rem}
		The sufficient part of Theorem \ref{classicalstri} had been obtained in  Guo--Li--Nakanishi--Yan \cite{guo2018boundary}. See also Guo--Peng--Wang \cite{guo2008decay}. For counterexamples regarding the boundary cases, see Montgomery-Smith \cite{montgomery1998time} for $(d,\alpha) = (2,2)$ or $(3,1)$, Fang--Wang \cite{fang2006some} for $(d,\alpha,q) = (2,1,4)$, Guo--Li--Nakanishi--Yan \cite{guo2018boundary} for $0<\alpha\leq 2$, $(q,r) = (2,\infty)$. The new result here is the counterexamples when $\alpha>2$.
	\end{rem}
	
	For \eqref{stritypeesti}, we have
	\begin{thm}\label{mainresbourgain}
		Let $d\in \mathbb{N}$, $0<q,r\leq \infty$, $(\alpha,s,b)\in \mathbb{R}^3$. Then \eqref{stritypeesti} holds iff
		$2\leq q,r\leq \infty$, $b\geq 1/2-1/q$, $(q,b)\neq (\infty,1/2)$, $(r,s,b) \neq (\infty,d/2,1/2-1/q)$, $(q,r,s)\neq (\infty,\infty,d/2)$, $(2,\infty,d/2-|\alpha|b)$, $(2,\infty,(d-|\alpha|)/2)$,
		$$s\geq \frac{d}{2}-\frac{|\alpha|}{q}-\frac{d}{r},~ s+|\alpha| b\geq \frac{d+|\alpha|}{2}-\frac{|\alpha|}{q}-\frac{d}{r},~(s,b)\neq \left(\frac{d}{2}-\frac{|\alpha|}{q}-\frac{d}{r},\frac{1}{2}\right),$$
		and
		\begin{itemize}
			\item $s\geq d/2-d/r$, $(r,s)\neq (\infty,d/2)$, or
			\item $\alpha>0$, $\alpha\neq 1$, $s\geq (1-\alpha/2)(d/2-d/r)$, or
			\item $\alpha = 1$, $d\geq 2$,  $(d,r,s)\neq (2,\infty,3/4)$, $s\geq (d+1)/4-(d+1)/(2r)$.
		\end{itemize}
	\end{thm}
	
	\begin{rem}
		Such estimates to the Schr\"{o}dinger equation were obtained by the first author in \cite{chen2024strichartz}. Although the result when $\alpha<0$ in the above theorem is trivial, there are still non-trivial Strichartz estimates \eqref{clastri}, with reference to the results in \cite{cho2011remarks}.
	\end{rem}
	
	\textbf{Notations}. For brevity, we denote $L_t^q(\mathbb{R};L_x^r(\mathbb{R}^d))$ as $L_t^qL_x^r$. For $b>0$, we use $a\lesssim b$ to denote that there exists a constant $C$ which may rely on $d,\alpha, q,r,s,b$ only such that $a\leq Cb$, and if $C$ is sufficiently large, we write it as $a\ll b$. We also denote $a\lesssim b$ as $b\gtrsim a$, $a\ll b$ as $b\gg a$, and $a\lesssim b\lesssim a$ as $a\sim b$.
	
	Let $\langle x \rangle = (1+|x|^2)^{1/2}$. Define $D^s = \mathscr{F}^{-1}|\xi|^s\mathscr{F}$, $J^s= \mathscr{F}^{-1} \langle \xi\rangle^s \mathscr{F}$, and $\|u_0\|_{\dot{H}^s}:=\||\xi|^s\widehat{u_0}(\xi)\|_{L^2}$, $\|u_0\|_{H^s}:=\|\langle\xi\rangle^s\widehat{u_0}(\xi)\|_{L^2}$. We use $\mathscr{F}$ to denote the Fourier transform with respect to spatial variable $x$, and $\mathscr{F}_{t,x}$ to denote the Fourier transform with respect to space-time variables.
	
	We always use $N,L$ to denote dyadic numbers. Fixed a radially decreasing $\varphi\in C_0^\infty(\mathbb{R}^d)$, $\varphi|_{B(0,1/2)} \equiv 1$, $\varphi|_{B(0,3/4)^c}\equiv 0$, let $\varphi_N(\xi) = \varphi(\xi/N)$, $\psi_N = \varphi_{2N}-\varphi_N$. Define the inhomogeneous Littlewood--Paley projection $P_N$ by $\mathscr{F}^{-1}\psi_N\mathscr{F}$ for $N\geq 2$ and $P_1 = \mathscr{F}^{-1}\varphi\mathscr{F}$. Let $P_{<N} = \mathscr{F}^{-1}\varphi_N\mathscr{F}$, $P_{>N} = \mathscr{F}^{-1}(1-\varphi_N)\mathscr{F}$. Also we define the modulation decomposition by
	$$Q_1^\alpha:= \mathscr{F}_{\tau,\xi}^{-1}\varphi(\tau-|\xi|^\alpha)\mathscr{F}_{t,x},\quad Q_L^\alpha:= \mathscr{F}_{\tau,\xi}^{-1}\psi_L(\tau-|\xi|^\alpha)\mathscr{F}_{t,x}, ~L\geq 2.$$

	\section{Some counterexamples to boundary Strichartz estimates}
	In this section, we show Theorem \ref{classicalstri}. In fact, all cases except $(q,r) = (2,\infty)$, $d\geq 3$, $\alpha>2$ are already known. See for example \cite{montgomery1998time,fang2006some,guo2008decay}.
	
	For $d = 1, 2$, we have the following result which is a simple application of Tao's argument in \cite{tao2006counterexample}. We give it here for completeness.
	\begin{thm}\label{xinftyd12}
		Let $\Phi(\xi) \in C^2(B\rightarrow \mathbb{R})$, $0\neq \phi \in C_0^\infty(B)$ where $B$ is an open set in $\mathbb{R}^d$(We extend $\Phi$ to be $0$ outside $B$). Then for $d = 1$, $q < 4$ or $d = 2$, $q\leq 2$, there does not have a $C>0$ for which one has
		\begin{equation}\label{d12endcon}
			\left\|\int_{\mathbb{R}^d}e^{it\Phi(\xi)+ix\cdot \xi}\phi(\xi)f(\xi)~d\xi\right\|_{L_t^qL_x^\infty}\leq C\|f\|_{L^2(\mathbb{R}^d)},\quad \forall~f\in L^2(\mathbb{R}^d).
		\end{equation}
	\end{thm}
	\begin{proof}[\textbf{Proof}]
		By the Bernstein inequality, \eqref{d12endcon} holds for $q = \infty$. Thus by the H\"{o}lder inequality, we only need to consider $q\geq 2$. Assume that \eqref{d12endcon} holds. By translation, we would assume that $0\in B$ and $\phi(0)\neq 0$. Let $\tilde{\phi}\in C_0^\infty (B)$ and $\tilde{\phi}|_{\mathrm{supp}(\phi)}\equiv 1$. Note that
		\begin{align*}
			&\quad\left\|\int_{\mathbb{R}^d}e^{it\Phi(\xi)+ix\cdot \xi}\phi(\xi)f(\xi)~d\xi\right\|_{L_x^\infty} \\
			& = \left\|\int_{\mathbb{R}^d}e^{it(\Phi(\xi)-\Phi(0)-\nabla\Phi(0)\cdot \xi)+ix\cdot \xi}\phi(\xi)f(\xi)~d\xi\right\|_{L_x^\infty}
		\end{align*}
		Without loss of generality, we would assume $\Phi(0) = 0$, $\nabla \Phi(0) = 0$. Since $\|\mathscr{F}^{-1}(\tilde{\phi}(\xi)e^{-i\theta \Phi(\xi)})\|_{L^1}\lesssim\|\tilde{\phi}(\xi)e^{-i\theta \Phi(\xi)}\|_{H^2_\xi}\lesssim 1$ uniformly for $0\leq \theta\leq 1$. Then
		\begin{align*}
			&\quad\left\|\int_{\mathbb{R}^d}e^{it\Phi(\xi)+ix\cdot \xi}\phi(\xi)f(\xi)~d\xi\right\|_{L^\infty_x}\\
			& = \left\|\int_{\mathbb{R}^d}\int_0^1e^{i(t+\theta)\Phi(\xi)+ix\cdot \xi}\phi(\xi)f(\xi)e^{-i\theta\Phi(\xi)}\tilde{\phi}(\xi)~d\theta d\xi\right\|_{L^\infty_x}\\
			& \leq \int_0^1\left\|\int_{\mathbb{R}^d}e^{i(t+\theta)\Phi(\xi)+ix\cdot \xi}\phi(\xi)f(\xi)e^{-i\theta\Phi(\xi)}\tilde{\phi}(\xi)~d\xi\right\|_{L^\infty_x} d\theta\\
			&\lesssim \left\|\int_{\mathbb{R}^d}e^{i(t+\theta)\Phi(\xi)+ix\cdot \xi}\phi(\xi)f(\xi)~d\xi\right\|_{L^q_\theta L^\infty_x}.
		\end{align*}
		Thus
		\begin{align*}
			&\quad\left\|\int_{\mathbb{R}^d}e^{in\Phi(\xi)+ix\cdot \xi}\phi(\xi)f(\xi)~d\xi\right\|_{l^q_nL^\infty_x}\\
			&\lesssim \left\|\int_{\mathbb{R}^d}e^{i(n+\theta)\Phi(\xi)+ix\cdot \xi}\phi(\xi)f(\xi)~d\xi\right\|_{l_n^qL_\theta^qL_x^\infty}\\
			&\sim \left\|\int_{\mathbb{R}^d}e^{it\Phi(\xi)+ix\cdot \xi}\phi(\xi)f(\xi)~d\xi\right\|_{L_t^qL_x^\infty}\\
			&\lesssim \|f\|_{L^2(\mathbb{R}^d)}.
		\end{align*}
		Then by duality one has
		\begin{equation*}
			\left\|\sum_{n=-N}^N e^{in\Phi(\xi)+ix_n\cdot \xi}\phi(\xi)c_n \right\|_{L^2_\xi}\lesssim \|c_n\|_{l^{q'}_n}
		\end{equation*}
		for any $x_n$. By choosing $c_n = 1$, $|n|\leq N$, $c_n = 0$ for $|n|>N$, we conclude that
		\begin{align*}
			\int_{\mathbb{R}^d}|\phi(\xi)|^2\sum_{n=-N}^N\sum_{n' = -N}^N e^{i(n-n')\Phi(\xi)+i(x_n-x_{n'})\cdot \xi}~d\xi\lesssim N^{\frac{2}{q'}}.
		\end{align*}
		Let $\gamma>0$, $x_{-N} = 0$, $x_{n+1}-x_n$, $-N\leq n<N$ be  independent and identically Gaussian random variable  on $\mathbb{R}^d$ with variation characteristic function $e^{-\gamma |\xi|^2}$. By integration on random variables we obtain that for any $\gamma>0$,
		\begin{equation}\label{discrete}
			\int_{\mathbb{R}^d}|\phi(\xi)|^2\sum_{n=-N}^N\sum_{n' = -N}^N e^{i(n-n')\Phi(\xi)-\gamma |n-n'||\xi|^2}~d\xi\lesssim N^{\frac{2}{q'}}.
		\end{equation}
		By scaling and the dominated convergence theorem we have
		\begin{align*}
			&\quad\lim_{y\rightarrow\infty} |y|^{\frac{d}{2}} \mathrm{Re} \int_{\mathbb{R}^d}|\phi(\xi)|^2 e^{iy\Phi(\xi)-\gamma |y||\xi|^2}~d\xi \\
			& =  \int_{\mathbb{R}^d}\lim_{y\rightarrow\infty}|\phi(|y|^{-\frac{1}{2}}\xi)|^2\cos(|y|\Phi(|y|^{-\frac{1}{2}}\xi))e^{-\gamma |\xi|^2}~d\xi\\
			& = \int_{\mathbb{R}^d}|\phi(0)|^2\cos\frac{(\nabla^2\Phi(0)\xi,\xi)}{2}e^{-\gamma |\xi|^2}~d\xi.
		\end{align*}
		By choosing $\gamma\gg 1$, the former integral is positive. We denote its value as $c$. Thus there exists $M$ such that for any $|y|\geq M$, we have 
		$$\mathrm{Re} \int_{\mathbb{R}^d}|\phi(\xi)|^2 e^{iy\Phi(\xi)-\gamma |y||\xi|^2}d\xi\geq \frac{c}{2}|y|^{-\frac{d}{2}}.$$
		Then by \eqref{discrete} for any $N>M$ we have
		\begin{align*}
			\sum_{|n|,|n'|\leq N, |n-n'|\geq M}^N\frac{c}{2|n-n'|^{d/2}}\lesssim N^{\frac{2}{q'}}+NM.
		\end{align*}
		The left side is asymptotically $N^{3/2}$ for $d = 1$ and $N\log N$ for $d = 2$ as $N\rightarrow \infty$, then we obtain $q\geq 4$ for $d = 1$ and $q>2$ for $d = 2$.
	\end{proof}
	
	\subsection{The \texorpdfstring{$L_t^2L_x^\infty$}{L2tLinftyx} estimate when \texorpdfstring{$\alpha>2$}{alpha>2}}
	We use the idea of Carbery and
	Hofmann. This method can avoid the use of stochastic processes. The key point is that this endpoint estimate implies the embedding from the H\"{o}lder space to the Triebel--Lizorkin space, which in turn leads to a contradiction. See Montgomery-Smith \cite{montgomery1998time}. Firstly, we show the following elementary lemma, which is useful in controlling the remainder term.
	\begin{lemma}\label{aninteesti}
		Let $0<\gamma<1$, $h>0$.
		\begin{align*}
			\sup_{A>0}\left|\int_0^A \frac{\sin r}{r}(e^{-r^\gamma h}-1+r^\gamma h)~dr\right|\lesssim \min\{h,h^2+h^{\frac{1}{\gamma}}\}.
		\end{align*}
	\end{lemma}
	\begin{proof}[\textbf{Proof}]
		If $A\leq 1$, it is easy to obtain the bound $\min\{h,h^2\}$. If $A>1$, $\gamma\neq 1/2$, by integration by parts we have
		\begin{align*}
			&\quad\left|\int_0^A \frac{\sin r}{r}(e^{-r^\gamma h}-1+r^\gamma h)~dr\right|\\
			&\lesssim \min\{h,h^2\}+ \left|\int_1^A \frac{\sin r}{r}(e^{-r^\gamma h}-1+r^\gamma h)~dr\right|\\
			&\lesssim \min\{h,h^2\}+ |(e^{-h}-1+h)|+\frac{|e^{-A^\gamma h}-1+A^\gamma h|}{A}\\
			&\quad + \left|\int_1^A \frac{\cos r}{r^2}(e^{-r^\gamma h}-1+r^\gamma h)~dr\right|+ h\left|\int_1^A \frac{\cos r}{r^{2-\gamma}}(-e^{-r^\gamma h} + 1)~dr\right|\\
			&\lesssim \min\{h,h^2\}+ A^{-1}\min\{A^\gamma h,A^{2\gamma}h^2\}+\int_1^A\frac{\min\{r^\gamma h,(r^\gamma h)^2\}}{r^2}~dr\\
			&\lesssim \min\{h,h^2+h^{\frac{1}{\gamma}}\}.
		\end{align*}
		If $\gamma = 1/2$, we perform integration by parts again, and then conclude the proof.
	\end{proof}
	
	\begin{prop}\label{t2xinftyesti}
		Let $d \in \mathbb{N}$, $\alpha>0$. There does not have a constant $C_{d,\alpha}>0$ such that 
		$$\|D^{\frac{\alpha-d}{2}}e^{itD^\alpha}u_0\|_{L_t^2L_x^\infty}\leq C_{d,\alpha}\|u_0\|_{L^2}.$$
	\end{prop}
	\begin{proof}[\textbf{Proof}]
		For $\alpha\leq 2$ or $d = 1$, $2$, the result had been obtained in \cite{guo2018boundary}. Below we always assume $d\geq 3$, $\alpha>1$. We argue by contradiction.
		By duality, for any continuous $x(t):\mathbb{R}\rightarrow \mathbb{R}^d$, one has
		\begin{align*}
			\left|\int_{\mathbb{R}}\int_{\mathbb{R}^d}|\xi|^{\frac{\alpha-d}{2}}e^{it|\xi|^\alpha+ix(t)\cdot \xi}f(\xi) g(t)~d\xi dt\right|\leq C\|f\|_{L^2}\|g\|_{L^2}.
		\end{align*}
		Then
		\begin{equation*}
			\left\|\int_{\mathbb{R}}|\xi|^{\frac{\alpha-d}{2}}e^{it|\xi|^\alpha+ix(t)\cdot \xi} g(t)~ dt\right\|_{L^2_\xi}\leq C\|g\|_{L^2}.
		\end{equation*}
		By using spherical coordinates, we obtain
		\begin{equation}\label{discretever}
			\int_{\mathbb{S}^{d-1}}d\sigma(\gamma)\int_0^\infty \rho^{\alpha-1}\left|\int_{\mathbb{R}}e^{it \rho^\alpha+i \rho x(t)\cdot \gamma}g(t)~dt\right|^2 ~d\rho\leq C^2\|g\|_{L^2}^2.
		\end{equation}
		Let $x(t) = (p(t),0,\cdots,0)$ and $I$ be an interval, $g(t) = \chi_{I}(t)$. Then
		\begin{align*}
			\int_0^\infty \rho^{\alpha-1}~d\rho\int_{I}\int_{I}dt_1dt_2\int_{\mathbb{S}^{d-1}}e^{i(t_1-t_2) \rho^\alpha+i \rho (x(t_1)-x(t_2))\cdot \gamma}~d\sigma(\gamma) \leq C^2|I|.
		\end{align*}
		We compute the integral over the sphere (cf. Appendix D.3 in \cite{grafakos2014classical}). Then
		\begin{align*}
			\int_0^\infty \rho^{\alpha-1}~d\rho\int_{I}\int_{I}dt_1dt_2\int_{-1}^1e^{i(t_1-t_2) \rho^\alpha+i \rho |p(t_1)-p(t_2)| s}(1-s^2)^{\frac{d-3}{2}}~ds \leq C'|I|.
		\end{align*}
		Let $r = \rho^\alpha$. We conclude
		\begin{equation*}
			\int_0^\infty dr\int_{-1}^1(1-s^2)^{\frac{d-3}{2}}\left|\int_I e^{itr+ir^\frac{1}{\alpha}p(t)s}~dt\right|^2~ds\leq C' |I|.
		\end{equation*}
		Then
		\begin{align*}
			&\quad\int_1^2 \int_0^\infty\left|\int_I e^{istr+ir^\frac{1}{\alpha}p(t)}~dt\right|^2~drds\\
			&\lesssim \int_1^2 \int_0^\infty\left|\int_I e^{itr+is^{-\frac{1}{\alpha}}r^\frac{1}{\alpha}p(t)}~dt\right|^2~drds\\
			&\lesssim \int^{2^{-1}}_{2^{-\frac{1}{\alpha}-1}}\int_0^\infty\left|\int_I e^{itr+i2sr^\frac{1}{\alpha}p(t)}~dt\right|^2~drds\\
			&\lesssim\int_0^\infty dr\int_{-1}^1(1-s^2)^{\frac{d-3}{2}}\left|\int_I e^{itr+i2r^\frac{1}{\alpha}p(t)s}~dt\right|^2~ds\\
			& \lesssim |I|.
		\end{align*}
		Similarly we have
		$$\int_{-2}^{-1} \int_0^\infty\left|\int_I e^{istr+ir^\frac{1}{\alpha}p(t)}~dt\right|^2~drds \lesssim |I|.$$
		For any $N>0$, $p(t)$, we have
		\begin{align*}
			\int_{1<|s|<2}ds\int_0^Ndr \int_{I}\int_I \cos(s|t_1-t_2|r)e^{ir^{\frac{1}{\alpha}}(p(t_1)-p(t_2))}~ dt_1dt_2\lesssim |I|.
		\end{align*}
		Let $p(t) = yz(t)$ and $z(t)$ be Lipschitz continuous. Then
		\begin{align*}
			\int_{I}\int_I\frac{dt_1dt_2}{|t_1-t_2|}\int_0^{N|t_1-t_2|}  \frac{\sin(2r)-\sin r}{r}e^{iyr^{\frac{1}{\alpha}}\frac{z(t_1)-z(t_2)}{|t_1-t_2|^{1/\alpha}}}~dr \lesssim |I|.
		\end{align*}
		Let $0<p\leq 2$, $f_p:=\mathscr{F}(e^{-|\cdot|^p})$. We have $f_p\in L^1(\mathbb{R})$. See for example \cite{blumenthal1960some}. Thus we have
		\begin{align*}
			&\quad
			\int_{I}\int_I\frac{dt_1dt_2}{|t_1-t_2|}\int_0^{N|t_1-t_2|}  \frac{\sin(2r)-\sin r}{r}e^{-r^{\frac{p}{\alpha}}\frac{|z(t_1)-z(t_2)|^p}{|t_1-t_2|^{p/\alpha}}}~dr\\
			&\sim \int_{\mathbb{R}}f_p(y)~dy\int_{I}\int_I\frac{dt_1dt_2}{|t_1-t_2|}\int_0^{N|t_1-t_2|}  \frac{\sin(2r)-\sin r}{r}e^{iyr^{\frac{1}{\alpha}}\frac{z(t_1)-z(t_2)}{|t_1-t_2|^{1/\alpha}}}~dr\\
			& \lesssim |I|.
		\end{align*}
%
		By Lemma \ref{aninteesti}, for $p<\alpha$ there exists $\mu>1$ such that for any $h>0$,
		$$\sup_{A>0}\left|\int_0^A \frac{\sin (2r)-\sin r}{r}(e^{-r^{\frac{p}{\alpha}}h}-1+r^{\frac{p}{\alpha}}h)~dr\right|\lesssim h^\mu.$$
		Thus we have
		\begin{align*}
			&\quad\int_{I}\int_I\frac{dt_1dt_2}{|t_1-t_2|}\int_0^{N|t_1-t_2|}  \frac{\sin r-\sin(2r)}{r}r^\frac{p}{\alpha}\frac{|z(t_1)-z(t_2)|^p}{|t_1-t_2|^{p/\alpha}}~dr\\
			& \lesssim |I| + \left|	\int_{I}\int_I\frac{dt_1dt_2}{|t_1-t_2|}\int_0^{N|t_1-t_2|}  \frac{\sin(2r)-\sin r}{r}~dr\right|\\
			&\quad + 	\int_{I}\int_I\frac{1}{|t_1-t_2|}
			\left(\frac{|z(t_1)-z(t_2)|^p}{|t_1-t_2|^{p/\alpha}}\right)^\mu ~dt_1dt_2.
		\end{align*}
		Note that
		\begin{align*}
			&\quad\int_{I}\int_I\frac{dt_1dt_2}{|t_1-t_2|}\int_0^{N|t_1-t_2|}  \frac{\sin(2r)-\sin r}{r}~dr\\
			& = \int_{I}\int_I\frac{dt_1dt_2}{|t_1-t_2|}\int_{N|t_1-t_2|}^{2N|t_1-t_2|}  \frac{\sin r}{r}~dr\\
			& = 2\int_0^{|I|}dt_1\int_0^{t_1}\frac{dt_2}{t_2}\int_{Nt_2}^{2Nt_2}\frac{\sin r}{r}~dr\\
			& = 2\int_0^{|I|}dt_1\left(\int_0^{Nt_1}\log 2\frac{\sin r}{r}~dr+\int_{Nt_1}^{2Nt_1}\frac{\sin r}{r}~dr\int_{r/2N}^{t_1}\frac{dt_2}{t_2}\right).
		\end{align*}
		Since $\sup_A|\int_0^A \sin r/r~dr|<\infty$, we obtain
		\begin{align*}
			\left|	\int_{I}\int_I\frac{dt_1dt_2}{|t_1-t_2|}\int_0^{N|t_1-t_2|}  \frac{\sin(2r)-\sin r}{r}~dr\right|\lesssim |I|.
		\end{align*}
		By taking $N\rightarrow \infty$, we conclude
		\begin{equation}\label{pert}
			\begin{aligned}
				&\quad\int_{I}\int_I\frac{|z(t_1)-z(t_2)|^p}{|t_1-t_2|^{1+p/\alpha}}~dt_1dt_2\int_0^{\infty}  \frac{\sin r-\sin(2r)}{r}r^\frac{2}{\alpha}~dr\\
				& \lesssim |I| + 	\int_{I}\int_I\frac{1}{|t_1-t_2|}\left(\frac{|z(t_1)-z(t_2)|^p}{|t_1-t_2|^{p/\alpha}}\right)^\mu~dt_1dt_2.
			\end{aligned}
		\end{equation}
		Note that
		\begin{align*}
			\int_0^\infty\frac{\sin r-\sin(2r)}{r}r^\frac{2}{\alpha}~dr &= (1-2^{-\frac{2}{\alpha}})\int_0^\infty \frac{\sin r}{r^{1-2/\alpha}}~dr\\
			& = \frac{1-2^{-\frac{2}{\alpha}}}{\Gamma(1-2/\alpha)}\int_0^\infty \sin r\int_0^\infty u^{-\frac{2}{\alpha}} e^{-ru}~dudr\\
			& = \frac{1-2^{-\frac{2}{\alpha}}}{\Gamma(1-2/\alpha)}\int_0^\infty \frac{u^{-\frac{2}{\alpha}}}{1+u^2} ~du>0.
		\end{align*}
		Let $0<h_0\ll 1$. For any $z$ with $\|z\|_{\dot{C}^{1/\alpha}}\leq h_0$, by \eqref{pert} we conclude
		\begin{equation}\label{averest}
			\int_{I}\int_I\frac{|z(t_1)-z(t_2)|^p}{|t_1-t_2|^{1+p/\alpha}}~dt_1dt_2\lesssim |I|.
		\end{equation}
		By \eqref{averest} and the definition 3.3 in Yang--Yuan \cite{yang20082760}, we obtain
		$$\|z\|_{Q^{\frac{1}{\alpha},p}_\infty}:=\sup_I \left(\frac{1}{|I|}\int_{I}\int_I\frac{|z(t_1)-z(t_2)|^p}{|t_1-t_2|^{1+p/\alpha}}~dt_1dt_2\right)^\frac{1}{p}\lesssim 1\lesssim \|z\|_{\dot{C}^{\frac{1}{\alpha}}}.$$
		However by Corollary 3.1 and Proposition 3.2 in \cite{yang20082760}, we have
		$$Q^{\frac{1}{\alpha},p}_\infty = \dot{F}^{\frac{1}{\alpha}}_{\infty,p}\varsubsetneq \dot{F}^{\frac{1}{\alpha}}_{\infty,\infty} = \dot{C}^\frac{1}{\alpha}.$$
		We obtain a contradiction and finish the proof.
	\end{proof}
	\begin{rem}
		The result here is new only when $\alpha>2$. If $\alpha>2$, we can choose $p=2$ in the proof and use the earlier result by Strichartz in \cite{strichartz1980bounded}, which asserts that $Q^{s,2}_\infty = I^s(BMO)= \dot{F}^s_{\infty,2}$, $0<s<1$.
	\end{rem}

	\begin{rem}
		Let $d\in \mathbb{N}$, $\alpha>0$. Below we give a new proof when $0<\alpha\leq 2$. In fact, through careful integration by parts one has 
		$$\sup_{0\leq \theta\leq 1}\left|\int_{\mathbb{R}^d}e^{i\theta |\xi|^\alpha}\varphi(\xi)e^{ix\cdot \xi}~d\xi\right|\leq C_{d,\alpha}\langle x\rangle^{-(d+\alpha)}.$$
		Following the argument in Theorem \ref{xinftyd12}, if $\|e^{it D^\alpha}u_0\|_{L^2_tL_x^\infty}\lesssim \|u_0\|_{\dot{H}^{(d-\alpha)/2}}$ one has
		\begin{align*}
			\int_{\mathbb{R}^d}|\xi|^{\alpha-d}|\varphi(\xi)|^2\sum_{n=-N}^N\sum_{n' = -N}^N e^{i(n-n')|\xi|^\alpha+i(x_n-x_{n'})\cdot \xi}~d\xi\lesssim N.
		\end{align*}
		For $0<\alpha\leq 2$, we use independent and identically distributed random variables with characteristic function $e^{-\gamma |\xi|^\alpha}$ (cf. \cite{blumenthal1960some} and reference therein) and obtain
		$$\int_{\mathbb{R}^d}|\xi|^{\alpha-d}|\varphi(\xi)|^2\sum_{n=-N}^N\sum_{n' = -N}^N \cos(|n-n'||\xi|^\alpha) e^{-\gamma |n-n'| |\xi|^\alpha}~d\xi\lesssim N.$$
		For $\gamma\gg 1$ there exists $c>0$ such that for any $M\gg 1$
		$$\int_{\mathbb{R}^d}|\xi|^{\alpha-d}|\varphi(\xi)|^2\cos(M|\xi|^\alpha) e^{-\gamma M |\xi|^\alpha}~d\xi\geq \frac{c}{M}.$$
		Thus for $N\gg 2M$, we obtain $N\log N\lesssim NM$, which is a contradiction by taking $N$ tend to $\infty$.
	\end{rem}
	
	Although the endpoint Strichartz estimates are not true, the inequalities hold if we first integrate with respect to $t$. When $\alpha = 1$, see Proposition 4 in \cite{fang2006some}.	
	\begin{lemma}
		Let $d\geq 1$, $\alpha\neq 0$, $2\leq q<\infty$. Then for any $\mathscr{F}u_0\in C_0^\infty(\mathbb{R}^d\setminus \{0\})$, we have
		$$\|D^{\frac{\alpha}{q}-\frac{d}{2}}e^{itD^\alpha} u_0\|_{L_x^\infty L_t^q}\lesssim \|u_0\|_{L^2}.$$
	\end{lemma}
	\begin{proof}[\textbf{Proof}]
		By the Minkowski inequality, the Sobolev inequality, and the H\"{o}lder inequality, we have
		\begin{align*}
			&\quad\|D^{\frac{\alpha}{q}-\frac{d}{2}}e^{itD^\alpha} u_0\|_{L_x^\infty L_t^q}\\
			& \sim \left\|\int_{\mathbb{R}^d}|\xi|^{\frac{\alpha}{q}-\frac{d}{2}}e^{it|\xi|^\alpha}\widehat{u_0}(\xi)e^{ix\cdot \xi}~d\xi\right\|_{L_x^\infty L_t^q}\\
			&\sim \left\|\int_{\mathbb{S}^{d-1}}\int_0^\infty\rho^{\frac{\alpha}{q}+\frac{d}{2}-1}e^{it \rho^\alpha}\widehat{u_0}(\rho \gamma)e^{i \rho x\cdot \gamma}~d\rho d\sigma (\gamma)\right\|_{L_x^\infty L_t^q}\\
			&\lesssim \int_{\mathbb{S}^{d-1}}\left\|\int_0^\infty r^{\frac{1}{q}+\frac{d}{2\alpha}-1}e^{it  r}\widehat{u_0}(r^{\frac{1}{\alpha}} \gamma)e^{i r^{\frac{1}{\alpha}} x\cdot \gamma}~d r \right\|_{L_x^\infty L_t^q}d\sigma (\gamma)\\
			&\sim  \int_{\mathbb{S}^{d-1}}\|r^{\frac{d-\alpha}{2\alpha}} \widehat{u_0}(r^{\frac{1}{\alpha}} \gamma)\|_{L_{r>0}^2}d\sigma (\gamma)\\
			&\sim  \int_{\mathbb{S}^{d-1}}\|\rho^{\frac{d-1}{2}} \widehat{u_0}(\rho \gamma)\|_{L_{\rho>0}^2}d\sigma (\gamma)\\
			&\lesssim \|u_0\|_{L^2}.
		\end{align*}
		We conclude the proof.
	\end{proof}
	For the sufficiency part of Theorem \ref{classicalstri}, refer to Chapter 3 of \cite{wang2011harmonic} and \cite{guo2018boundary}. By Proposition \ref{t2xinftyesti} and the counterexamples in \cite{montgomery1998time,fang2006some,guo2018boundary}, we obtain the necessity part of Theorem \ref{classicalstri} and conclude the proof of this theorem.

	\section{Necessary conditions for the Strichartz type estimates}
	In this section, we show the necessity part of Theorem \ref{mainresbourgain}. Firstly, we show the trivial case $\alpha\leq 0$ or $(d,\alpha) = (1,1)$.
	
	\begin{prop}\label{alphanetrivial}
		Let $d\in \mathbb{N}$, $\alpha<0$ or $(d,\alpha) = (1,1)$. If \eqref{stritypeesti} holds, then $q, r\geq 2$, $b\geq 1/2-1/q$, $s\geq d/2-d/r$, $(q,b)\neq 
		(\infty,1/2)$, and $(r,s) = (\infty,d/2)$.
	\end{prop}
	\begin{proof}[\textbf{Proof}]		
		If $\alpha\leq 0$, we have $\|u\|_{X_\alpha^{s,b}}\sim \|P_1u\|_{X^{0,b}_\alpha}+\|J^s_xJ^b_tP_{>1}u\|_{L_{t,x}^2}$. Then $\|P_{>2}(f(t)g(x))\|_{L_t^qL_x^r}\lesssim \|J_x^sJ^b_tP_{>2}(f(t)g(x))\|_{L_{t,x}^2}$. Thus we obtain 
		$$\|f\|_{L_t^q}\lesssim \|f\|_{H^b_t},\quad \|P_{>2}g\|_{L_x^r}\lesssim \|g\|_{H^s_x}.$$
		By necessary conditions of the Sobolev inequality, we conclude the proof.
		
		If $(d,\alpha) = (1,1)$, for any $u\in X^{s,b}_1$ with $\hat{u}(\tau,\xi) = 0$ for $\xi<0$, we have
		\begin{align*}
			\|u\|_{L_t^qL_x^r} = \|u(t,x+t)\|_{L_t^qL_x^r}&\lesssim\|\langle\xi\rangle^s\langle\tau-|\xi|\rangle^b\hat{u}(\tau-\xi,\xi)\|_{L^2_{\tau,\xi}}\\
			&\sim \|\langle\xi\rangle^s\langle\tau\rangle^b\hat{u}(\tau,\xi)\|_{L^2_{\tau,\xi}}.
		\end{align*}
		We conclude the proof by the same argument as for $\alpha\leq 0$.
	\end{proof}
	
	By the same argument as for Lemmas 2.1, 2.2 in \cite{chen2024strichartz}, we have
	\begin{prop}\label{basicrestrction}
		Given $d \in \mathbb{N}$, $\alpha>0$, if \eqref{stritypeesti} holds for some $0<q,r\leq \infty$, $s,b\in \mathbb{R}$, then $q,r\geq 2$, 
		$b\geq 1/2-1/q$, $(b,q)\neq (1/2,\infty)$, $(r,s,b)\neq (\infty,d/2,1/2-1/q)$, $(q,r,s)\neq (\infty,\infty,d/2)$, $(2,\infty,d/2-\alpha b)$ and
		$$s\geq \frac{d}{2}-\frac{\alpha}{q}-\frac{d}{r},~ s+\alpha b\geq \frac{d+\alpha}{2}-\frac{\alpha}{q}-\frac{d}{r},~(s,b)\neq \left(\frac{d}{2}-\frac{\alpha}{q}-\frac{d}{r},\frac{1}{2}\right).$$
	\end{prop}
	\begin{proof}[\textbf{Sketch of proof}]
		Let $0\neq \phi\in C_0^\infty((-1,1)\setminus (-1/2,1/2))$, $\phi\geq 0$. $N\geq 4$.
		
		Let
		$$u_N = \mathscr{F}^{-1}_{\tau,\xi}(\phi(\tau/N^\alpha)\psi(\xi/N)).$$
		Then $\|u_N\|_{X^{s,b}_\alpha}\sim N^{s+\alpha b+(d+\alpha)/2}$ and $\|u_N\|_{L_t^qL_x^r}\sim N^{d/r'+\alpha/q'}$. Let $N$ tend to infinity, and we obtain $s+\alpha b\geq (d+\alpha)/2-\alpha/q-d/r$.
		
		Let
		$$v_N = \mathscr{F}^{-1}_{\tau,\xi}(\phi(\tau-|\xi|^\alpha)\psi(\xi/N)).$$
		Then $\|v_N\|_{X^{s,b}_\alpha}\sim N^{s+d/2}$. Since for $|t|\ll N^{-\alpha}$, $|x|\ll N^{-1}$, one has $|v_N(t,x)|\sim N^d$. Thus $\|v_N\|_{L_t^qL_x^r}\gtrsim N^{d-\alpha/q-d/r}$. Let $N$ tend to infinity, and we obtain $s\geq d/2-\alpha/q-d/r$.
		
		If for some $2\leq q,r\leq \infty$, one has $\|u\|_{L_t^qL_x^r}\lesssim \|u\|_{X^{d/2-\alpha/q-d/r,1/2}_\alpha}$. Let
		$$w_N = \mathscr{F}^{-1}_{\tau,\xi}\left(\frac{\psi(\xi/N)\chi_{[1,N^\alpha/2 ]}(\tau-|\xi|^\alpha)}{\log (N^\alpha/(\tau-|\xi|^\alpha))(\tau-|\xi|^\alpha)}\right).$$ 
		Then $\|w_N\|_{X^{d/2-\alpha/q-d/r,1/2}_\alpha}\sim N^{d-\alpha/q-d/r}$. By the Bernstein inequality, we have $\|w_N\|_{L_t^qL_x^r}\gtrsim N^{-d/r}\|w_N\|_{L_t^qL_x^\infty}$. For $|t|\ll N^{-\alpha}$, one has 
		\begin{align*}
			\|w_N(t,x)\|_{L_x^\infty}\geq |w_N(t,0)|&\gtrsim \int_{\mathbb{R}^{1+d}}\frac{\psi(\xi/N)\chi_{[1,N^\alpha/2 ]}(\tau-|\xi|^\alpha)}{\log (N^\alpha/(\tau-|\xi|^\alpha))}\frac{d\tau d\xi}{\tau-|\xi|^\alpha}\\
			&\gtrsim N^d\log\log N.
		\end{align*}
		Thus $\|w_N\|_{L_t^qL_x^r}\gtrsim N^{d-\alpha/q-d/r}\log\log N$. Let $N$ tend to infinity, we obtain a contradiction. Thus $(s,b)\neq (d/2-\alpha/q-d/r,1/2)$. For other cases, the argument in \cite{chen2024strichartz} is fully applicable, so we omit it.
	\end{proof}

	In Propositions \ref{anotherboundsalphanot1}--\ref{anotherboundsalpha1}, we show another lower bound of $s$. See \cite{chen2025strichartz} for a similar argument to the Airy equation.
	\begin{prop}\label{anotherboundsalphanot1}
		Given $d \in \mathbb{N}$, $\alpha>0$, $\alpha\neq 1$, if \eqref{stritypeesti} holds, then $s\geq (1-\alpha/2)(d/2-d/r)$.
	\end{prop}
	\begin{proof}[\textbf{Proof}]
		Let $\phi\in C_0^\infty(\mathbb{R})$, $\phi|_{[0,1]}\equiv 1$, $\gamma >-1$, $N\gg 1$ and
		$$u_N(t,x) = \phi(t)e^{itD^\alpha}\mathscr{F}^{-1}_\xi(\psi(N^{\gamma}(\xi-Ne_1)))$$
		where $e_1 = (1,0,\cdots,0)$. Then $\|u_N\|_{X^{s,b}_\alpha}\sim N^{s-\gamma d/2}$ and
		\begin{align*}
			\|u_N\|_{L_t^qL_x^r}&\sim \left\|\phi(t)\int_{\mathbb{R}^d}\psi(N^{\gamma} \xi)e^{ix\cdot  \xi+it|\xi+Ne_1|^\alpha}~d\xi\right\|_{L_t^qL_x^r}\\
			&\sim N^{-\gamma d+\frac{\gamma d}{r}+\frac{\gamma \alpha}{q}} \left\|\phi(N^{\gamma \alpha}t)\int_{\mathbb{R}^d}\psi(\xi)e^{ix\cdot  \xi+it|\xi+N^{1+\gamma}e_1|^\alpha}~d\xi\right\|_{L_t^qL_x^r}.
		\end{align*}
		By the Taylor expansion, one has
		\begin{align*}
			&\quad|\xi+N^{1+\gamma}e_1|^\alpha \\
			& = N^{\alpha(1+\gamma)} + \alpha N^{(\alpha-1)(1+\gamma)}\xi_1+ \left(\frac{\alpha}{2}(\alpha-1)\xi_1^2+\frac{\alpha}{2}|\xi'|^2\right)N^{(\alpha-2)(1+\gamma)}+G(\xi).
		\end{align*}
		where $|\partial^\alpha_\xi G(\xi)|\lesssim N^{(\alpha - 3)(1+\gamma)}$ for any $\alpha\in \mathbb{N}_0^d$.
		
		If $|x+\alpha N^{(\alpha-1)(1+\gamma)}e_1|\gg \langle N^{(1+\gamma)(\alpha-2)}t\rangle$, by the non-stationary argument (see for example Chapter VIII, \cite{stein1993harmonic}) we have
		$$\left|\int_{\mathbb{R}^d}\psi(\xi)e^{ix\cdot  \xi-it|\xi+N^{1+\gamma}e_1|^\alpha}~d\xi\right|\lesssim \langle x\rangle^{-d-1}.$$
		By the stationary argument (see again \cite{stein1993harmonic}) we have
		$$\left|\int_{\mathbb{R}^d}\psi(\xi)e^{ix\cdot  \xi-it|\xi+N^{1+\gamma}e_1|^\alpha}~d\xi\right|\lesssim \langle N^{(1+\gamma)(\alpha-2)}t\rangle^{-\frac{d}{2}}.$$
		Thus
		\begin{align*}
			\left\|\int_{\mathbb{R}^d}\psi(\xi)e^{ix\cdot  \xi-it|\xi+N^{1+\gamma}e_1|^\alpha}~d\xi\right\|_{L^1_x}\lesssim \langle N^{(1+\gamma)(\alpha-2)}t\rangle^{\frac{d}{2}}.
		\end{align*}
		Then
		\begin{align*}
			\|u_N\|_{L_t^qL_x^r}\gtrsim N^{-\gamma d+\frac{\gamma d}{r}+\frac{\gamma \alpha}{q}}\|\phi(N^{\gamma \alpha }t)\langle N^{(1+\gamma)(\alpha-2)}t\rangle^{-\frac{d}{2}+\frac{d}{r}}\|_{L_t^q}.
		\end{align*}
		By choosing $\gamma = (\alpha-2)/2$, we obtain
		$$N^{\frac{d(2-\alpha)}{2r'}}\lesssim \|u_N\|_{L_t^qL_x^r}\lesssim \|u_N\|_{X^{s,b}_\alpha}\sim N^{s+\frac{d(2-\alpha)}{4}}.$$
		As $N$ tends to infinity, we conclude that $s\geq  (1-\alpha/2)(d/2-d/r)$.
	\end{proof}
	
	\begin{prop}\label{anotherboundsalpha1}
		Given $d\in \mathbb{N}$, $\alpha = 1$, if \eqref{stritypeesti} holds, then $s\geq (d+1)(1/2-1/r)/2$.
	\end{prop}
	\begin{proof}[\textbf{Proof}]
		We only need to consider $d\geq 2$. Let $\phi\geq 0$, $\phi|_{[-1/4,1/4]} = 1$, $\phi\in C_0^\infty(-1/2,1/2)$, 
		$\xi = (\xi_1,\xi')$, $x = (x_1,x')$. We define
		$$u_N(t,x) = \phi(t)\int_{\mathbb{R}^d}\phi(N^{-1}(\xi_1-N))\phi(N^{-\frac{1}{2}}|\xi'|)e^{ix\cdot \xi+it|\xi|}~d\xi.$$
		Then we have
		$$\|u_N\|_{X^{s,b}_1}\sim N^s\|\phi(N^{-1}(\xi_1-N))\phi(N^{-\frac{1}{2}}|\xi'|)\|_{L^2_\xi}\sim N^{s+\frac{d+1}{4}}$$
		and
		\begin{align*}
			&\quad\|u_N\|_{L_t^qL_x^r}\\
			&= N^{\frac{d+1}{2}(1-\frac{1}{r})}\left\| \phi(t)\int_{\mathbb{R}^d}\phi(\xi_1-1)\phi(|\xi'|)e^{ix\cdot \xi+it\sqrt{N^2\xi_1^2+N|\xi'|^2}}~d\xi\right\|_{L_t^qL_x^r}\\
			& = N^{\frac{d+1}{2}(1-\frac{1}{r})}\left\| \phi(t)\int_{\mathbb{R}^d}\phi(\xi_1-1)\phi(|\xi'|)e^{ix\cdot \xi+it\frac{|\xi'|^2}{\sqrt{\xi_1^2+N^{-1}|\xi'|^2}+\xi_1}}~d\xi\right\|_{L_t^qL_x^r}.
		\end{align*}		
		Note that for $|x|,|t|\ll 1$, we have
		\begin{align*}
			&\quad\left|\int_{\mathbb{R}^d}\phi(\xi_1-1)\phi(|\xi'|)e^{ix\cdot \xi+it\frac{|\xi'|^2}{\sqrt{\xi_1^2+N^{-1}|\xi'|^2}+\xi_1}}~d\xi\right|\\
			& \sim \int_{\mathbb{R}^d}\phi(\xi_1-1)\phi(|\xi'|)~d\xi\sim 1.
		\end{align*}
		We obtain $\|u_N\|_{L_t^qL_x^r}\gtrsim N^{(d+1)(1-1/r)/2}$, then $s\geq (d+1)(1/2-1/r)/2$.
	\end{proof}
	
	Finally, we give some counterexamples related to the endpoint Strichartz estimates.
	\begin{prop}
		If \eqref{stritypeesti} holds for some $d\in \mathbb{N}$, $\alpha>0$, then $(q,r,s)\neq (2,\infty,(d-\alpha)/2)$.
	\end{prop}
	\begin{proof}[\textbf{Proof}]
		Let $\phi\in C_0^\infty(\mathbb{R})$, $\phi|_{[0,1]}\equiv 1$. If \eqref{stritypeesti} holds for $(q,r,s) = (2,\infty,(d-\alpha)/2)$, then
		$$\|e^{itD^\alpha}P_{>1}u_0\|_{L_t^2(0,1; L^\infty_x)}\lesssim \|\phi(t)e^{itD^\alpha}P_{>1}u_0\|_{X^{s,b}_\alpha}\lesssim \|u_0\|_{\dot{H}^{\frac{d-\alpha}{2}}}.$$
		Thus $\|D^{\frac{\alpha-d}{2}}P_{>1}e^{itD^\alpha}u_0\|_{L^2_t(0,1;L^\infty_x)}\lesssim \|u_0\|_{L^2}$. Let $u_{0,\lambda}(\cdot) = \lambda^{d/2}u_0(\lambda\cdot)$. We obtain
		\begin{align*}
			&\quad\|D^{\frac{\alpha-d}{2}}P_{>1}e^{itD^\alpha}u_{0,\lambda}\|_{L^2_t(0,1;L^\infty_x)}\\
			& \sim \left\|\int_{\mathbb{R}^d}|\xi|^{\frac{\alpha-d}{2}}(1-\varphi(\lambda\xi))e^{it |\xi|^\alpha+ix\cdot \xi}\widehat{u_0}(\xi)~d\xi\right\|_{L_t^2(0,\lambda^\alpha;L^\infty_x)}.
		\end{align*}
		Due to $\alpha>0$, we conclude that
		\begin{align*}
			\|D^{\frac{\alpha-d}{2}}e^{itD^\alpha}u_0\|_{L_t^2L_x^\infty}\lesssim \limsup_{\lambda\rightarrow\infty}\|D^{\frac{\alpha-d}{2}}P_{>1}e^{itD^\alpha}u_{0,\lambda}\|_{L^2_t(0,1;L^\infty_x)}\lesssim \|u_0\|_{L^2}.
		\end{align*}
		By Proposition \ref{t2xinftyesti}, we obtain a contradiction and conclude the proof.
	\end{proof}
	
	\begin{prop}\label{wave2cont}
		Let $d = 2$, $r = \infty$. If \eqref{stritypeesti} holds then $s>3/4$.
	\end{prop}
	\begin{proof}[\textbf{Proof}]
		We argue by contradiction. By Proposition \ref{basicrestrction} and the H\"{o}lder inequality, if \eqref{stritypeesti} holds for some $s\leq 3/4$, then $\|e^{itD}u_0\|_{L_t^1(0,1;L_x^\infty)}\lesssim \|u_0\|_{H^{3/4}}$. We use the argument in \cite{fang2006some}. Let $e_1 = (1,0)$. We have
		\begin{equation*}
			\left|\int_0^1 J_x^{-\frac{3}{4}}e^{itD}u_0(te_1)~dt\right|\leq \|J_x^{-\frac{3}{4}}e^{itD}u_0\|_{L_t^1(0,1;L_x^\infty)}\lesssim \|u_0\|_{L^2}.
		\end{equation*}
		The left side is
		\begin{equation*}
			\left|\int_{\mathbb{R}^2}\widehat{u_0}(\xi)\int_0^1e^{it(|\xi|+\xi_1)}\langle\xi\rangle^{-\frac{3}{4}}~dtd\xi\right| = \left|\int_{\mathbb{R}^2}\widehat{u_0}(\xi)\frac{e^{i(|\xi|+\xi_1)}-1}{|\xi|+\xi}\langle\xi\rangle^{-\frac{3}{4}}~d\xi\right|.
		\end{equation*}
		By duality and the Plancherel identity we obtain
		$$\left\|\frac{e^{i(|\xi|+\xi_1)}-1}{|\xi|+\xi_1}\langle\xi\rangle^{-\frac{3}{4}}\right\|_{L^2_\xi}<\infty.$$
		However let $\lambda_1 = |\xi|+\xi_1$, $\lambda_2 = \xi_2$, $\lambda = (\lambda_1,\lambda_2)$, we have
		\begin{align*}
			\left\|\frac{e^{i(|\xi|+\xi_1)}-1}{|\xi|+\xi_1}\langle\xi\rangle^{-\frac{3}{4}}\right\|_{L^2_\xi}&\sim \left\|\frac{e^{i\lambda_1}-1}{\lambda_1}\langle|\lambda|^2/\lambda_1\rangle^{-\frac{3}{4}}\frac{|\lambda|}{\lambda_1}\right\|_{L^2_{\lambda_1>0,\lambda_2}}\\
			&\gtrsim \left\|\frac{e^{i\lambda_1}-1}{\lambda_1}\lambda_2^{-\frac{1}{2}}\lambda_1^{-\frac{1}{4}}\right\|_{L^2_{\lambda_2>\lambda_1>1}} = \infty.
		\end{align*}
		We obtain a contradiction and conclude the proof.
	\end{proof}

	Combining Propositions \ref{alphanetrivial}--\ref{wave2cont}, we obtain the necessity of the conditions listed in Theorem \ref{mainresbourgain}.
	
	\section{Sufficient conditions for the Strichartz type estimates}
	In this section, we show the sufficiency part of Theorem \ref{mainresbourgain}. Firstly, by the Sobolev inequality, the Plancherel identity, and the Minkowski inequalities, one has the following trivial estimate. This argument also works for general dispersive equations.
	\begin{prop}\label{strtypetrivial}
		Let $d\in \mathbb{N}$, $\alpha\in \mathbb{R}$, $q,r\geq 2$, $b\geq 1/2-1/q$, $s\geq d/2-d/r$, $(b,q)\neq (1/2,\infty)$, $(s,r)\neq (d/2,\infty)$. Then \eqref{stritypeesti} holds.
	\end{prop}
	\begin{proof}[\textbf{Proof}]
		By the Sobolev inequality we have $\|u\|_{L_t^qL_x^r}\lesssim \|J^s u\|_{L_t^qL_x^2}$. Then due to the $L^2$ norm is conserved for dispersive semigroups and the Minkowski inequality, we have
		\begin{align*}
			\|u\|_{L_t^qL_x^r}\lesssim \|J^se^{-itD^\alpha}u\|_{L_t^qL_x^2}&\lesssim \|J^se^{-itD^\alpha}u\|_{L_x^2L_t^q}\\
			&\lesssim \|J^se^{-itD^\alpha}u\|_{L_x^2H_t^b}\sim \|u\|_{X^{s,b}_\alpha}.
		\end{align*}
		We conclude the proof.
	\end{proof}

	We show the following interpolation lemma, which is useful to obtain the sharp region of $(s,b)$ when $r<\infty$.
	\begin{lemma}\label{intepolemma}
		Let $q_1,q_2, r_1,r_2\geq 2$, $r<\infty$, $2/q = 1/q_1+1/q_2$, $2/r = 1/r_1+1/r_2$. If for some $s_1,s_2$, $b_1\neq b_2$, we have $\|P_NQ_L^\alpha u\|_{L_t^{q_j}L_x^{r_j}}\lesssim N^{s_j}L^{b_j}\|u\|_{L^2_{t,x}}$, $j = 1,2$ for any $N, L\geq 1$, then $\|u\|_{L_t^qL_x^r}\lesssim \|u\|_{X^{s,b}_\alpha}$ where $s = (s_1+s_2)/2$, $b = (b_1+b_2)/2$.
	\end{lemma}
	\begin{proof}[\textbf{Proof}]
		Without loss of generality, we assume that $b_1>b_2$. Then
		\begin{align*}
			\|P_Nu\|_{L_t^qL_x^r}^2 & \lesssim \sum_{L_1\leq L_2}\|Q^\alpha_{L_1}P_Nu\|_{L_t^qL_x^r}\|Q^\alpha_{L_2}P_Nu\|_{L_t^{q_1}L_x^{r_1}}\\
			&\lesssim N^{2s}\sum_{L_1\leq L_2} (L_1 L_2)^{b} \left(\frac{L_1}{L_2}\right)^{\frac{b_1-b_2}{2}}\|Q^\alpha_{L_1}P_Nu\|_{L_{t,x}^2}\|Q^\alpha_{L_2}P_Nu\|_{L_{t,x}^2}\\
			&\lesssim \|P_Nu\|_{X_\alpha^{s,b}}^2.
		\end{align*}
		Then by the Littlewood--Paley theory, we conclude the proof.
	\end{proof}
	
	\begin{prop}\label{stri}
		Let $d\in \mathbb{N}$, $\alpha>0$, $\alpha\neq 1$, $2\leq q,r<\infty$, the conditions in Theorem \ref{mainresbourgain} are sufficient.
	\end{prop}
	\begin{proof}[\textbf{Proof}]
		If $2/q+d/r\leq d/2$, by the Strichartz estimate and Proposition \ref{strtypetrivial}, we have
		$$\|Q_L^\alpha P_N u\|_{L_t^qL_x^r}\lesssim L^\frac{1}{2}N^{\frac{d}{2}-\frac{\alpha}{q}-\frac{d}{r}}\|Q_L^\alpha P_Nu\|_{L_{t,x}^2}.$$
		and
		$$\|Q_L^\alpha P_N u\|_{L_t^qL_x^r}\lesssim L^{\frac{1}{2}-\frac{1}{q}}N^{\frac{d}{2}-\frac{d}{r}}\|Q_L^\alpha P_Nu\|_{L_{t,x}^2}.$$
		Then by the interpolation for $0\leq \theta\leq 1$ we have
		$$\|Q_L^\alpha P_N u\|_{L_t^qL_x^r}\lesssim (L^\frac{1}{2}N^{\frac{d}{2}-\frac{\alpha}{q}-\frac{d}{r}})^{\theta}(L^{\frac{1}{2}-\frac{1}{q}}N^{\frac{d}{2}-\frac{d}{r}})^{1-\theta}\|Q_L^\alpha  P_Nu\|_{L_{t,x}^2}.$$		
		By Proposition \ref{strtypetrivial} and Lemma \ref{intepolemma}, we obtain $\|u\|_{L_t^qL_x^r}\lesssim \|u\|_{X^{s,b}_\alpha}$ for $1/2-1/q\leq b<1/2$, $s+\alpha b = (d+\alpha)/2-\alpha/q-d/r$.
		
		If $2/q+d/r>d/2$, by Proposition \ref{strtypetrivial} we only need to show $\|u\|_{L_t^qL_x^r}\lesssim \|u\|_{X_\alpha^{s_{\alpha,r},b_{q,r}}}$ where $s_{\alpha,r} = (1-\alpha/2)(d/2-d/r)$, $b_{q,r} = 1/2-1/q+(d/2-d/r)/2$. By the Strichartz estimate and Proposition \ref{strtypetrivial}, for $r = 2$ or $2/q+d/r = d/2$ ($2\leq q,r<\infty$) one has
		$$\|P_NQ_L^\alpha u\|_{L_t^qL_x^r}\lesssim N^{s_{\alpha,r}}L^{b_{q,r}}\|u\|_{L^2_{t,x}}.$$
		Then by Lemma \ref{intepolemma}, if $2/q+d/r>d/2$, $\|u\|_{L_t^qL_x^r}\lesssim \|u\|_{X_\alpha^{s_{\alpha,r},b_{q,r}}}$ holds for $d = 1$ \& $q\geq 4$ or $d = 2$ \& $q>2$ or $d \geq 3$, $q\geq 2$.
		
		For $d = 1, 2$, we use the method in \cite{chen2024strichartz}. Choose a Schwartz function $\phi$ such that, $\mathrm{supp}(\hat{\phi})\subset [-1,1]$. $\phi|_{[-1,1]}\geq \chi_{[-1,1]}$. Define $\phi_{L,k}(t):=\phi(Lt-k)$, $k\in \mathbb{Z}$.
		\begin{equation}\label{deco}
			\begin{aligned}
				\|P_NQ_L^\alpha u\|_{L_t^2L_x^r}^2&\leq \sum_{k} \|\phi_{L,k}(t)(Q^\alpha _{L}P_Nu)\|_{L_{t\in [kL^{-1},(k+1)L^{-1}]}^2L_x^r}^2\\
				&\lesssim \sum_k L^{\frac{d}{2}-\frac{d}{r}-1}\|\phi_{L,k}(t)(Q_{L}^\alpha P_Nu)\|_{L_t^{\frac{4r}{d(r-2)}}L_x^r}^2\\
				&\lesssim \sum_k L^{\frac{d}{2}-\frac{d}{r}}N^{(2-\alpha)(\frac{d}{2}-\frac{d}{r})}\|\phi_{L,k}(t)Q_{L}^\alpha P_Nu\|_{L_{t,x}^2}^2\\
				&\sim N^{2s_{\alpha,r}}L^{2b_{2,r}} \|u\|_{L_{t,x}^2}^2.
			\end{aligned}
		\end{equation}
		Then by Lemma \ref{intepolemma}, we conclude the proof.
	\end{proof}
	
	\begin{prop}\label{alphaneq1rinfty}
		Let $d\in \mathbb{N}$, $\alpha>0$, $\alpha\neq 1$, $r = \infty$, the conditions in Theorem \ref{mainresbourgain} are sufficient.
	\end{prop}
	\begin{proof}[\textbf{Proof}]
		Firstly, we consider the case $q>2$. We use the argument in \cite{chen2024strichartz}. Let $0<\epsilon\ll 1$, $b = 1/2-1/q+\epsilon$, $s = (d+\alpha)/2-\alpha/q-\alpha b$. By the Bernstein inequality
		\begin{equation*}
			\begin{aligned}
				&\quad\|Q_{L_1}^\alpha uQ_{L_2}^\alpha u\|_{L_t^{\frac{q}{2}}L_x^\infty}\\
				&\lesssim \sum_{N\lesssim M}\|Q^\alpha_{L_1}P_N u\|_{L_t^{\frac{q}{1-2q\epsilon}}L_x^\infty}\|Q^\alpha_{L_2}P_M u\|_{L_t^{\frac{q}{1+2q\epsilon}}L_x^\infty}\\
				&\quad+\sum_{N\gg M}\|Q^\alpha_{L_1}P_N u\|_{L_t^qL_x^\infty}\|Q^\alpha_{L_2}P_M u\|_{L_t^qL_x^\infty}\\
				&\lesssim\sum_{N\lesssim M}N^{\frac{d}{2}}M^{\frac{d(1+2q\epsilon)}{q}}\|Q^\alpha_{L_1}P_N u\|_{L_t^{\frac{q}{1-2q\epsilon}}L_x^2}\|Q^\alpha_{L_2}P_M u\|_{L_{t,x}^{\frac{q}{1+2q\epsilon}}}\\
				&\quad+\sum_{N\gg M}M^{\frac{d}{2}}N^{\frac{d}{q}}\|Q^\alpha_{L_1}P_N u\|_{L_{t,x}^{q}}\|Q^\alpha_{L_2}P_M u\|_{L_t^qL_x^2}.
			\end{aligned}
		\end{equation*}
		By Proposition \ref{stri}, for any $0<\varepsilon\ll1$, we have
		$$\|Q_L^\alpha P_Nu\|_{L_{t,x}^q}\lesssim L^{\frac{1}{2}-\frac{1}{q}+\varepsilon}N^{\frac{d}{2}-\frac{d}{q}-\varepsilon \alpha}\|Q_L^\alpha P_Nu\|_{L_{t,x}^2}$$
		and $\|Q_{L}^\alpha P_N u\|_{L_t^qL_x^2}\lesssim L^{1/2-1/q}\|Q_{L}^\alpha P_N u\|_{L_{t,x}^2}$. Then we obtain
		\begin{equation}\label{bilinearinfi}
			\begin{aligned}
				&\quad\|Q_{L_1}^\alpha uQ_{L_2}^\alpha u\|_{L_t^{\frac{q}{2}}L_x^\infty}\\
				&\lesssim\sum_{N\lesssim M}N^{\frac{d}{2}}M^{\frac{d}{2}-2\alpha \epsilon}(L_1L_2)^{b}\left(\frac{L_1}{L_2}\right)^{\epsilon}\|Q^\alpha_{L_1}P_N u\|_{L_{t,x}^2}\|Q_{L_2}^\alpha P_M u\|_{L_{t,x}^{2}}\\
				&\quad+\sum_{N\gg M}M^{\frac{d}{2}}N^{\frac{d}{2}-2\epsilon\alpha}(L_1L_2)^{b}\left(\frac{L_1}{L_2}\right)^{\epsilon}\|Q^\alpha_{L_1}P_N u\|_{L_{t,x}^2}\|Q^\alpha_{L_2}P_M u\|_{L_{t,x}^2}\\
				&\lesssim \left(\frac{L_1}{L_2}\right)^{\epsilon}\|Q^\alpha_{L_1}u\|_{X_\alpha^{\frac{d}{2}-\epsilon \alpha,b}}\|Q^\alpha_{L_2}u\|_{X_\alpha^{\frac{d}{2}-\epsilon \alpha,b}}.
			\end{aligned}
		\end{equation}		
		If $d = 1$ \& $q\geq 4$ or $d\geq 2$, by the Strichartz estimate we have
		$$\|Q_L^\alpha u\|_{L_t^qL_x^\infty}\lesssim \|Q_L^\alpha u\|_{X^{\frac{d}{2}-\frac{\alpha}{q},\frac{1}{2}}_\alpha}.$$
		
		By interpolation we conclude that for any $1/2-1/q<b<1/2$, there exists $\delta>0$ such that
		$$\|Q_{L_1}^\alpha uQ_{L_2}^\alpha u\|_{L_t^{\frac{q}{2}}L_x^\infty}\lesssim \left(\frac{L_1}{L_2}\right)^{\delta}\|Q^\alpha_{L_1}u\|_{X_\alpha^{\frac{d+\alpha}{2}-\frac{\alpha}{q}-\alpha b,b}}\|Q^\alpha_{L_2}u\|_{X_\alpha^{\frac{d+\alpha}{2}-\frac{\alpha}{q}-\alpha b,b}}.$$
		Thus we obtain that \eqref{stritypeesti} holds for $r = \infty$, $d = 1$ \& $q\geq 4$ or $d\geq 2$ \& $q>2$, $s\geq d/2-\alpha/q$, $b\geq 1/2-1/q$, $s+\alpha b\geq (d+\alpha)/2-\alpha/q$, $(s,b)\neq (d/2-\alpha/q,1/2)$, and $(s,b)\neq (d/2,1/2-1/q)$.
		
		If $d = 1$, $r = \infty$, $2<q<4$, by \eqref{deco}, we also have
		\begin{equation}\label{timedec}
			\begin{aligned}
				\|Q_L^\alpha u\|_{L_t^2L_x^\infty}^2&\lesssim \sum_k L^{-\frac{1}{2}}\|\phi_{L,k}(t)(Q_{L}^\alpha  u)\|_{L_t^4L_x^\infty}^2\\
				&\lesssim \sum_k L^{\frac{1}{2}}\|\phi_{L,k}(t)(J^{\frac{1}{2}-\frac{\alpha}{4}}Q_{L}^\alpha  u)\|_{L_{t,x}^2}^2 \lesssim \|Q_L^\alpha u\|_{X_\alpha^{\frac{1}{2}-\frac{\alpha}{4}, \frac{1}{4}}}^2.
			\end{aligned}
		\end{equation}
		By interpolation we have $\|Q_L^\alpha u\|_{L_t^qL_x^\infty}\lesssim \|Q_L^\alpha u\|_{X^{1/2-\alpha/4,3/4-1/q}_\alpha}$.
		Then
		\begin{align*}
			\|Q^\alpha_{L_1}uQ^\alpha_{L_2}u\|_{L_t^{\frac{q}{2}}L_x^\infty} &\lesssim \|Q^\alpha_{L_1}u\|_{L^{\frac{q}{1-\epsilon}}_tL_x^\infty}\|Q^\alpha_{L_2}u\|_{L_t^\frac{q}{1+\epsilon}L_x^\infty}\\
			&\lesssim \left(\frac{L_1}{L_2}\right)^\frac{\epsilon}{q}\|Q^\alpha_{L_1}u\|_{X_\alpha^{\frac{1}{2}-\frac{\alpha}{4},\frac{3}{4}-\frac{1}{q}}}\|Q_{L_2}^\alpha u\|_{X_\alpha^{\frac{1}{2}-\frac{\alpha}{4},\frac{3}{4}-\frac{1}{q}}}.
		\end{align*}
		Thus
		\begin{align*}
			\|u\|^2_{L_t^qL_x^\infty}\lesssim \sum_{L_1\leq L_2}\|Q^\alpha_{L_1}uQ^\alpha_{L_2}u\|_{L_t^{\frac{q}{2}}L_x^\infty} \lesssim \|u\|_{X_\alpha^{\frac{1}{2}-\frac{\alpha}{4},\frac{3}{4}-\frac{1}{q}}}^2.
		\end{align*}
		By the interpolation and \eqref{bilinearinfi}, we obtain $\|u\|_{L_t^qL_x^\infty}\lesssim \|u\|_{X_\alpha^{s,b}}$, $s+\alpha b = (1+\alpha)/2-\alpha/q$, $1/2-1/q<b\leq 3/4-1/q$. Also, from \eqref{timedec} we have $\|u\|_{L_t^2L_x^\infty}\lesssim \|u\|_{X^{1/4-\alpha/4,b}_\alpha}$ for any $b$ larger than $1/4$.
		
		If $q = 2$, $d\geq 2$, by the Sobolev inequality and Proposition \ref{stri}, we conclude the proof.
	\end{proof}

	Combining Propositions \ref{strtypetrivial}--\ref{alphaneq1rinfty}, we obtain Theorem \ref{mainresbourgain} except $\alpha = 1$. The case $\alpha = 1$ is similar, except for some different indices.
		
	\begin{prop}\label{except2ifnty}
		Let $d\in \mathbb{N}$, $\alpha=1$, $2\leq q,r<\infty$, the conditions in Theorem \ref{mainresbourgain} are sufficient.
	\end{prop}
	\begin{proof}[\textbf{Proof}]
		Similar to the proof of Proposition \ref{stri}, by Proposition \ref{strtypetrivial}, Lemma \ref{intepolemma}, and the Strichartz estimate, all cases can be proven, except $(d,q) = (3,2)$ or $d = 2$, $2\leq q\leq 4$.
		 
		Let $s_{r} = (d+1)(1/2-1/r)/2$, $b_{q,r} = 1/2-1/q+(d-1)(1/2-1/r)/2$. For $d = 2, 3$, we also use the argument in \cite{chen2024strichartz}. Choose a Schwartz function $\phi$ such that, $\mathrm{supp}(\hat{\phi})\subset [-1,1]$. $\phi|_{[-1,1]}\geq \chi_{[-1,1]}$. Define $\phi_{L,k}(t):=\phi(Lt-k)$, $k\in \mathbb{Z}$.
		\begin{equation}\label{timecut}
			\begin{aligned}
				\|P_NQ_L^1 u\|_{L_t^2L_x^r}^2&\leq \sum_{k} \|\phi_{L,k}(t)(Q^1 _{L}P_Nu)\|_{L_{t\in [kL^{-1},(k+1)L^{-1}]}^2L_x^r}^2\\
				&\lesssim \sum_k L^{\frac{d-1}{2}-\frac{d-1}{r}-1}\|\phi_{L,k}(t)(Q_{L}^1 P_Nu)\|_{L_t^{\frac{4r}{(d-1)(r-2)}}L_x^r}^2\\
				&\lesssim \sum_k L^{\frac{d-1}{2}-\frac{d-1}{r}}N^{2s_r}\|\phi_{L,k}(t)Q_{L}^1 P_Nu\|_{L_{t,x}^2}^2\\
				&\sim L^{2b_{2,r}} N^{2s_{r}}\|Q_L^1 P_N u\|_{L_{t,x}^2}^2.
			\end{aligned}
		\end{equation} 
		By the Strichartz estimate and Lemma \ref{intepolemma}, we conclude the proof.
	\end{proof}
	
	\begin{prop}\label{alpha1rinfty}
		Let $d\in \mathbb{N}$, $\alpha = 1$, $r = \infty$, the conditions in Theorem \ref{mainresbourgain} are sufficient.
	\end{prop}
	\begin{proof}[\textbf{Proof}]
		By Proposition \ref{strtypetrivial}, we would assume that $d\geq 2$. By the Sobolev inequality and Proposition \ref{except2ifnty}, we only need to consider $q > 2$.
		
		Let $0<\epsilon\ll 1$, $b = 1/2-1/q+\epsilon$, $s = (d+1)/2-1/q- b$. Similar to the argument in Proposition \ref{alphaneq1rinfty}, by the Bernstein inequality and Propositions \ref{strtypetrivial}, \ref{except2ifnty}, one has
		\begin{equation}\label{bilinearinfi1}
			\begin{aligned}
				&\quad\|Q_{L_1}^1 uQ_{L_2}^1 u\|_{L_t^{\frac{q}{2}}L_x^\infty}\\
				&\lesssim\sum_{N\lesssim M}N^{\frac{d}{2}}M^{\frac{d}{2}-2 \epsilon}(L_1L_2)^{b}\left(\frac{L_1}{L_2}\right)^{\epsilon}\|Q^1_{L_1}P_N u\|_{L_{t,x}^2}\|Q_{L_2}^1 P_M u\|_{L_{t,x}^{2}}\\
				&\quad+\sum_{N\gg M}M^{\frac{d}{2}}N^{\frac{d}{2}-2\epsilon}(L_1L_2)^{b}\left(\frac{L_1}{L_2}\right)^{\epsilon}\|Q^1_{L_1}P_N u\|_{L_{t,x}^2}\|Q^1_{L_2}P_M u\|_{L_{t,x}^2}\\
				&\lesssim \left(\frac{L_1}{L_2}\right)^{\epsilon}\|Q^1_{L_1}u\|_{X_1^{\frac{d}{2}-\epsilon ,b}}\|Q^1_{L_2}u\|_{X_1^{\frac{d}{2}-\epsilon,b}}.
			\end{aligned}
		\end{equation}
		
		If $d = 2$ \& $q > 4$ or $d\geq 3$, by the Strichartz estimate we have
		$$\|Q_L^1 u\|_{L_t^qL_x^\infty}\lesssim \|Q_L^1u\|_{X^{\frac{d}{2}-\frac{1}{q},\frac{1}{2}}_1}.$$
		
		By interpolation, we conclude that for any $1/2-1/q<b<1/2$, there exists $\delta>0$ such that
		$$\|Q_{L_1}^1 uQ_{L_2}^1 u\|_{L_t^{\frac{q}{2}}L_x^\infty}\lesssim \left(\frac{L_1}{L_2}\right)^{\delta}\|Q^1_{L_1}u\|_{X_1^{\frac{d+1}{2}-\frac{1}{q}- b,b}}\|Q^1_{L_2}u\|_{X_1^{\frac{d+1}{2}-\frac{1}{q}-b,b}}.$$
		Thus we obtain that \eqref{stritypeesti} holds for $r = \infty$, $d = 2$ \& $q > 4$ or $d\geq 3$ \& $q>2$, $s\geq d/2-1/q$, $b\geq 1/2-1/q$, $s+b\geq (d+1)/2-1/q$, $(s,b)\neq (d/2-1/q,1/2)$, and $(s,b)\neq (d/2,1/2-1/q)$.
		
		Now we show that for $d = 2$, $2<q\leq 4$, $3/4<s<1$, $b = 3/2-1/q-s$, \eqref{stritypeesti} holds. By the Bernstein inequality, we have $\|Q_L^1P_Nu\|_{L_t^qL_x^\infty}\lesssim NL^{1/2-1/q}\|Q_L^1P_Nu\|_{L^2_{t,x}}$. By the Strichartz estimate one has 
		$$\|Q_{L}^1P_Nu\|_{L_t^4L_x^\infty}\lesssim L^{\frac{1}{2}}N^\frac{3}{4}\|Q_L^1P_Nu\|_{L^2_{t,x}}.$$
		By choosing $r = \infty$ in \eqref{timecut}, we obtain
		$$\|Q_L^1P_Nu\|_{L_t^2L_x^\infty}\lesssim L^\frac{1}{4}N^\frac{3}{4}\|Q_L^1P_Nu\|_{L^2_{t,x}}.$$
		Then by interpolation, we obtain
		$$\|Q_L^1P_Nu\|_{L_t^qL_x^\infty}\lesssim N^sL^b\|Q_L^1P_Nu\|_{L^2_{t,x}},\quad s+b = \frac{3}{2}-\frac{1}{q},\quad  \frac{3}{4}\leq s\leq 1.$$
		For any $2<q<4$, $0<\varepsilon\ll 1$,
		\begin{align*}
			&\quad\|u\|_{L_t^qL_x^\infty}^2\\
			&\lesssim \sum_{L_1\leq L_2}\sum_{N,M}\|Q_{L_1}^1P_{N}u\|_{L_t^{\frac{q}{1-\varepsilon}}L_x^\infty}\|Q^1_{L_2}P_Mu\|_{L^\frac{q}{1+\varepsilon}_tL_x^\infty}\\
			&\lesssim \sum_{L_1\leq L_2}\sum_{N\lesssim M}N^{\frac{3}{4}+\frac{\varepsilon}{q}}L_1^{\frac{3}{4}-\frac{1}{q}}M^\frac{3}{4}L_2^{\frac{3}{4}-\frac{1+\varepsilon}{q}}\|Q_{L_1}^1P_Nu\|_{L^2_{t,x}}\|Q_{L_2}^1P_Mu\|_{L^2_{t,x}}\\
			&\quad+\sum_{L_1\leq L_2}\sum_{N\gg M}N^{\frac{3}{4}}L_1^{\frac{3}{4}-\frac{1}{q}+\frac{\varepsilon}{q}}M^{\frac{3}{4}+\frac{\varepsilon}{q}}L_2^{\frac{3}{4}-\frac{1+2\varepsilon}{q}}\|Q_{L_1}^1P_Nu\|_{L^2_{t,x}}\|Q_{L_2}^1P_Mu\|_{L^2_{t,x}}\\
			&\lesssim \|u\|^2_{X^{\frac{3}{4}+\frac{\varepsilon}{2q},\frac{3}{4}-\frac{1}{q}-\frac{\varepsilon}{2q}}_1}.
		\end{align*}
		By interpolation and \eqref{bilinearinfi1}, we obtain the result for $2<q<4$. If $q = 4$, note that
		$$\|u\|_{L^\frac{4}{1-\varepsilon}_tL_x^\infty}\lesssim \|u\|_{X^{s,b}_1},\quad s+b = \frac{3}{2}-\frac{1-\varepsilon}{4},\quad \frac{3+\varepsilon}{4}<s<1$$
		and
		$$\|u\|_{L^\frac{4}{1+\varepsilon}_tL_x^\infty}\lesssim \|u\|_{X^{s,b}_1},\quad s+b = \frac{3}{2}-\frac{1+\varepsilon}{4},\quad \frac{3}{4}<s<1$$
		for any $0<\varepsilon <1$, by interpolation, we obtain $\|u\|_{L_t^4L_x^\infty}\lesssim \|u\|_{X_1^{5/4-b,b}}$, $1/4<b<1/2$ and conclude the proof.
	\end{proof}

	By Propositions \ref{strtypetrivial}, \ref{except2ifnty}, and \ref{alpha1rinfty}, we obtain Theorem \ref{stritypeesti} for $\alpha = 1$. We finish the proof of Theorem \ref{stritypeesti}.	
	
	\section*{Acknowledgments}
	The first author is supported in part by the NSFC, grants 12301116, 12450001. The authors thank Professors Baoxiang Wang and Zihua Guo for their invaluable support and encouragement, also thank Yufeng Lu for some helpful discussions.
	
	\phantomsection
	\bibliographystyle{amsplain}
	\addcontentsline{toc}{section}{References}
	\bibliography{reference}		
	
	\begin{enumerate}
		\item[] \scriptsize\textsc{Jie Chen: School of Science, Jimei University, Xiamen 361021, P.R. China}
		
		\textit{E-mail address}: \textbf{jiechern@jmu.edu.cn}
		
		\item[] \scriptsize\textsc{Yiran Gong: School of Science, Jimei University, Xiamen 361021, P.R. China}
		
		\textit{E-mail address}: \textbf{gyr@jmu.edu.cn}
		
		\item[] \scriptsize\textsc{Ying Zhang: Academy of Mathematics and Systems Science, CAS, Beijing 100190, P.R. China}
		
		\textit{E-mail address}: \textbf{zhangying24@amss.ac.cn}
	\end{enumerate}
\end{document}